\documentclass{article}

\usepackage{amsmath}
\usepackage{amssymb}
\usepackage{amsthm}
\usepackage[latin1]{inputenc}
\usepackage{bbm}

\usepackage{graphicx}
\usepackage{epstopdf}

\usepackage{ifthen}
\usepackage{color}

\renewcommand{\div}{\operatorname{div}}

\newcommand{\Tt}{{\mathbb{T}}}

\newcommand{\Pp}{\mathcal P}
 \newcommand{\Rr}{\mathbb R}
 \newcommand{\Zz}{\mathbb Z}

 \newcommand{\ep}{\epsilon}

\newcommand{\Ff}{\mathcal F}

\renewcommand{\div}{\operatorname{div}}

\newcommand{\bx}{{\bf x}}

\newcommand{\bv}{{\bf v}}

\newtheorem{teo}{Theorem}[section]

\newtheorem{Lemma}{Lemma}[section]
\newtheorem{Corollary}{Corollary}[section]
\newtheorem{Proposition}{Proposition}[section]

\newtheorem{Assumption}{A}

\begin{document}

\title{Time-dependent mean-field games with logarithmic nonlinearities}
\author{Diogo A. Gomes\footnote{King Abdullah University of Science and Technology (KAUST), CEMSE Division and
KAUST SRI, Center for Uncertainty Quantification in Computational Science and Engineering, Thuwal 23955-6900. Saudi Arabia. e-mail: diogo.gomes@kaust.edu.sa.}, 
Edgard Pimentel
\footnote{Universidade Federal do Cear\'a, Campus of Pici, Bloco 914, Fortaleza, Cear\'a 60.455-760, Brazil. e-mail: epimentel@mat.ufc.br.}
}

\date{\today} 

\maketitle

\begin{abstract}
In this paper, we prove the existence of classical solutions for time dependent mean-field games with a logarithmic nonlinearity and subquadratic Hamiltonians. Because the logarithm is unbounded from below, 
this nonlinearity poses substantial mathematical challenges that have not been addressed in the literature.
Our result is proven by recurring to a delicate argument, which combines Lipschitz regularity for the Hamilton-Jacobi equation with estimates for the nonlinearity in suitable Lebesgue spaces. Lipschitz estimates follow from an application of the nonlinear adjoint method. These are then combined with
 a priori bounds for solutions of the Fokker-Planck equation and a concavity argument for the nonlinearity.   
\end{abstract}

\thanks{
D. Gomes was partially supported by
KAUST baseline funds, KAUST SRI, Center for Uncertainty Quantification in Computational Science and Engineering.

E. Pimentel is financed by CNPq-Brazil.}

\section{Introduction}

In this paper, we discuss time dependent mean-field games with logarithmic nonlinearities.
More precisely, 
we study the following mean-field game (MFG) problem:
\begin{equation}\label{mfg}
\begin{cases}
-u_t+H(x,Du)=\Delta u+g[m]\\
m_t-\div(D_pH(x,Du)m)=\Delta m,
\end{cases}
\end{equation}
where $g[m]=\ln m$.  The previous system is endowed  
with initial-terminal boundary datum
\begin{equation}\label{itbc}
\begin{cases}
u(x,T)\,=\,u_T(x)\\
m(x,0)\,=\,m_0(x),
\end{cases}
\end{equation}
where the terminal time $T>0$ is fixed. For simplicity, we work in the periodic setting (i.e., the variable $x$ takes values in the $d$-dimensional flat torus $\Tt^d=\Rr^d\backslash \Zz^d$).
We consider Hamiltonians $H$ which satisfy subquadratic growth conditions, as well as a number of  additional hypotheses discussed in Section \ref{assumptions}.
A model Hamiltonian that satisfies these is
\begin{equation}
\label{eham}
H(x,p)=a(x) (1+|p|^2)^{\frac \gamma 2} +V(x), 
\end{equation}
with $a, V:\Tt^d\to \Rr$, $a, V\in C^\infty$, with $a>0$, and $1< \gamma <\frac{5}{4}$. As usual, we assume that $u_T$ and $m_0$ are smooth functions and that $m_0>0$.

As we explain in Section \ref{motivation}, 
the system \eqref{mfg}-\eqref{itbc} aims at modeling a strategic interaction problem
between a (very) large number of rational agents.
These optimize preferences that depend, among other things, on the density of the entire population. Here, we assume that this dependence is of logarithmic form (i.e., given by the nonlinearity $\ln(m(x,t))$).
This nonlinearity encodes a particular preference structure. More precisely, because $\ln(z)\to -\infty$ as $z\to 0$,  agents benefit from avoiding crowding and remaining in states where the density, $m$, is small.

Because the logarithm is unbounded from below, the existence of solutions was an open question until now. Indeed, for time dependent problems, earlier results in the literature rely on the fact that the nonlinearity has uniform lower bounds.
The results in \cite{gueant3} are an exception, where the author finds explicit 
Gaussian-Quadratic stationary solutions for 
\begin{equation}\label{mfgQ}
\begin{cases}
-u_t+\frac{|Du|^2}{2}=\Delta u+\ln m\\
m_t-\div(m Du)=\Delta m,
\end{cases}
\end{equation}
and studies their stability. Thanks to the specific logarithmic nonlinearity and the quadratic structure, the Hopf-Cole transformation is also a suitable tool for the study of \eqref{mfgQ}. This was done, in the stationary case, for instance, 
in \cite{GM} to prove directly the existence of solutions. The Hopf-Cole transformation was also used in 
\cite{MR2928382, MR2976439, MR2974160} to study time-dependent problems for nonlinearities $g$ bounded from below. In this paper, we do not use the Hopf-Cole transformation because we do not have a quadratic Hamiltonian. The advantage is that our results are more robust and can be easily modified for other nonlinearities with 
logarithmic-type behavior. 


The question of existence of solutions for mean-field game systems has been investigated by several authors since the seminal works of J-M. Lasry and P-L. Lions \cite{ll1,ll2,ll3} and M. Huang, P. Caines and R. Malham\'e \cite{Caines1, Caines2}. For an account of the recent developments in this direction see \cite{llg2}, \cite{cardaliaguet}, \cite{achdou2013finite}, or \cite{GS}, the lectures by P-L. Lions \cite{LCDF}, \cite{LIMA}, and the monograph by A. Bensousan, J. Frehse and P. Yam \cite{bensoussan}.

In previous works ,\eqref{mfg}-\eqref{itbc} has been considered under the hypothesis that the nonlinearity $g$ is bounded from below.  A typical choice is the power nonlinearity $g(z)=z^\alpha$, for some $\alpha>0$. See \cite{ll2}, \cite{ll3}, \cite{porretta}, \cite{porretta2}, \cite{GPM2}, \cite{GPM3}, \cite{GPim1}, \cite{GP2}, \cite{cgbt} and \cite{GPim1}.
Firstly, lower bounds for $g$ imply that the solutions of the Hamilton-Jacobi equation are bounded from below.
Then, because of the optimal control nature of the problem, estimates for solutions of the Hamilton-Jacobi equation in $L^\infty(\Tt^d\times[0,T])$ can be proved under various conditions. 
However, if $g$ fails to be bounded from below, those estimates are no longer valid. Moreover, the logarithmic structure of the nonlinearity produces further difficulties: the polynomial estimates for the solutions of the Fokker-Planck equation no longer lead to bounds for the nonlinearity $g$ in the appropriate Lebesgue spaces. In the stationary case, the logarithmic nonlinearity was investigated in \cite{GM}, \cite{GPM1}, and \cite{GPatVrt}. The stationary obstacle problem with power or logarithmic nonlinearity dependence was studied in
\cite{GPat}; the congestion problem was studied in \cite{GMit}.

To overcome the difficulties caused by the unboundedness of the logarithm,
we recur to the nonlinear adjoint method \cite{E3}.
This yields estimates for the solutions $u$ of the Hamilton-Jacobi equation in $L^\infty(\Tt^d\times[0,T])$ in terms of both norms of $g$ in Lebesgue spaces and the adjoint variable.
A further study of the regularity of the adjoint variable improves these estimates. Then, to obtain a bound for
$g$ in $L^\infty(0,T;L^p(\Tt^d))$, we explore the concavity properties of the logarithm combined with a priori bounds for
$\frac 1 m$, and use the Fokker-Planck equation. A further application of the adjoint method yields Lipschitz regularity for the solutions of the Hamilton-Jacobi equation in terms of norms of the nonlinearity $g$ in Lebesgue spaces. 

The main result of this paper is the following:
\begin{teo}\label{main}
Suppose that Assumptions A\ref{ham}-\ref{gamma} from Section \ref{assumptions} are satisfied. Then, there exists a classical solution $(u,m)$ for \eqref{mfg}, with $u$ and $m$ in $\mathcal{C}^\infty(\Tt^d\times[0,T])$ and $m\geq 0$, satisfying the initial-terminal boundary datum \eqref{itbc}.
\end{teo}

The Assumptions of the previous Theorem are presented in the next Section, where we outline its proof and the structure of this paper.

\section{Mean-field games theory \textemdash\, main assumptions and the outline of the proof}\label{ma}

We open this Section with an introduction to the theory of mean-field games and present an heuristic derivation of the MFG system \eqref{mfg}-\eqref{itbc}.

\subsection{A brief introduction to mean-field games}
\label{motivation}

The mean-field game theory investigates strategic interaction models
 involving a large number of agents.
 These models encode statistical information about the players in a mean field.  Individual 
 agents base their own decisions and strategies on this mean field, and collectively, drive the mean field 
 evolution. 
In the literature, this mean field is commonly encoded in a probability measure, which gives
the distribution of the agents. Each of the agents, taking into account this probability measure, 
computes a value function and an optimal strategy. This measure then evolves in time 
by aggregating the actions of the agents in a Fokker-Plack equation. 
In the sequel, we put forward an informal derivation of the model  \eqref{mfg}-\eqref{itbc}. 

For $t<T$, 
consider a controlled diffusion:
\begin{equation}\label{dif}
\begin{cases}
d\bx_s\,=\,{\bf v}ds+\sqrt{2}dW_s,\qquad t\leq s\leq T\\
\bx_t\,=\,x, 
\end{cases}
\end{equation}
where $W_t$ is a $d$-dimensional Brownian motion on a filtered probability space $(\Omega, \Ff_t, \Pp)$, 
${\bf v}$ is a $\Rr^d$ valued $\Ff_t$-progressively measurable control, 
and $x\in\Tt^d$ is a given initial condition. Let $J(x,\bv)$ be defined by 
\begin{equation}\label{cfunc}
J(x,t;\bv)=\int_t^TL(\bx_s,\bv_s)+r(\bx, t)ds+u_T(\bx_T),
\end{equation}
where $L:\Tt^d\times \Rr^d$ is the Lagrangian and $u_T$ is a terminal cost. Finally,  $r:\Tt^d\times [0,T]\to\Rr$ is a given function that we will specify later. 

 It is well known that, under general conditions,
 the value function associated with \eqref{dif}-\eqref{cfunc}
 given by
 $u(x,t)=\inf_{\bv} J(x, t, \bv)$ 
is a (viscosity) solution of 
\begin{equation}\label{hj1}
\begin{cases}
-u_t+H(X,Du(x,t))=\Delta u+r(x,t),\;\;\;&(x,t)\in\Tt^d\times[0,T)\\
u(x,T)=u_T(x)\;\;\;&x\in\Tt^d,
\end{cases}
\end{equation}
where $H(x,p)$ is the Legendre transform of the Lagrangian $L$. Furthermore, the optimal control is given in feedback form by 
\begin{equation}\label{optimal}
{\bf v}^*\,=\,-D_pH(x,Du).
\end{equation}

Assume that every single agent in the population faces the stochastic optimal control problem described by \eqref{dif}-\eqref{cfunc} and uses the optimal strategy.
Then, the density of agents evolves according to the Fokker-Planck equation:
\begin{equation}\label{FP2}
\begin{cases}
m_t-\div(D_pHm)=\Delta m,\;\;\;&(x,t)\in\Tt^d\times(0,T]\\
m(x,0)=m_0(x)\,=\,x,\;\;\;&x\in\Tt^d.
\end{cases}
\end{equation}
At this stage, the functional $J$ (and consequently $u$) does not depend on the density of agents. However, 
we can choose the function $r(x,t)=g(m(x,t))$, so that \eqref{hj1}-\eqref{FP2} becomes \eqref{mfg}. The specific choice
of $g(z)=\ln z$, penalizes agents that are in regions of high density and makes it very desirable to be in low-density areas. 

Clearly, low-density areas can be problematic for the regularity of solutions of \eqref{mfg}. 
A possible mechanism for the regularity of \eqref{mfg}, is that the diffusion will prevent low-density areas from 
arising. This would be immediate if the drift $D_pH$ were a priori with enough regularity. However, the regularity of $D_pH$ 
depends on bounds for $\ln m$. We explore this cross-dependency to simultaneously prove bounds for
$\ln m$ and $D_pH$. 

\subsection{Main assumptions}\label{assumptions}

We begin by introducing standard assumptions on the Hamiltonian $H$.
The first group of Assumptions,  A\ref{ham}-\ref{dxh}, are satisfied by a large class of Hamiltonians such as \eqref{eham},
but also by many other examples arising in practice.
They are introduced here for convenience and generality of the proof, but they are not overly restrictive. 
In contrast, Assumption A\ref{gamma} is
essential in our proof and imposes a suitable 
bound on the growth of the Hamiltonian (sublinear growth).
This sublinear growth condition corresponds to
 superlinear growth of the Lagrangian. This makes large drifts in \eqref{dif} expensive, and 
means that diffusion effects dominate the dynamics of the agents.

\begin{Assumption}\label{ham}
The Hamiltonian $H:\Tt^d\times\Rr^d\to\Rr$, $d>2$ is smooth and
\begin{enumerate}
\item for fixed $x$, the map $p\mapsto H(x,p)$ is strictly convex;
\item satisfies the coercivity condition $$\lim_{|p|\to\infty}\frac{H(x,p)}{|p|}\,=+\infty,$$and, without loss of generality, we suppose that in addition $H\geq 0$.
\item $H$ satisfies the growth condition: $$0\,\leq\,H(x,p)\,\leq\,C+C|p|^\gamma,$$ for some $\gamma> 1$, and $C>0$.
\end{enumerate}
The nonlinearity $g$ in \eqref{mfg} is $$g[m](x,t)\,\doteq\,\ln\left[m\right](x,t).$$
In addition, we suppose that $u_T, m_0\in C^\infty(\Tt^d)$ with $m_0>0$.
\end{Assumption}

Note that further conditions will be placed on the parameter $\gamma$ in Assumption A\ref{gamma}.




%
In this paper, we consider $d>2$. Nevertheless, minor modifications of our arguments yield similar results for the case $d\leq 2$.

Assumptions A\ref{dphminush}-\ref{dxh} that follow,
impose natural growth conditions on the first derivatives of the Hamiltonian $H$.

\begin{Assumption}\label{dphminush}There exists a constant $C>0$ such that
$$D_pH(x,p)p-H(x,p)\geq CH(x,p)-C.$$
\end{Assumption}

\begin{Assumption}\label{dphsq}There exists a constant $C>0$ such that
$$|D_pH|^2\leq C+C|Du|^{2(\gamma-1)}.$$
\end{Assumption}

\begin{Assumption}\label{dxh}There exists a constant $C>0$ such that
$$\left|D_xH(x,p)\right|\leq CH(x,p)+C.$$
\end{Assumption}

Lastly, we impose a condition on the exponent $\gamma$

\begin{Assumption}\label{gamma}
The exponent $\gamma$ satisfies $$1<\gamma<\frac{5}{4}.$$
\end{Assumption}

\subsection{Outline of the proof}\label{outline}

We consider an approximate problem whose limit solves \eqref{mfg}. This is done by replacing the operator $g[m]=\ln[m]$ with $$g_\ep[m](x,t)\,\doteq\,\ln[\epsilon+m](x,t).$$ As a consequence, we study the following regularized problem:
\begin{equation}\label{smfg}
\begin{cases}
-u^\ep_t+H(x,Du^\ep)=\Delta u^\ep+g_\ep[m^\ep]\\
m^\ep_t-\div(D_pH(x,Du)m^\ep)=\Delta m^\ep. 
\end{cases}
\end{equation} 
For fixed $\epsilon>0$, 
because $g_\epsilon$ is bounded from below and with logarithmic growth, 
the existence of classical solutions for \eqref{smfg}-\eqref{itbc} follows from standard arguments along the same lines of those in \cite{GPM2}.

To prove Theorem \ref{main}, Lipschitz regularity for $u^\ep$ plays a critical role. Indeed, we begin by
considering estimates for $g_\ep$ in terms of $Du^\ep$, as stated in the following Proposition:

\begin{Proposition}\label{regfp}
Let $(u^\ep,m^\ep)$ be a solution of \eqref{smfg}-\eqref{itbc} and assume that A\ref{ham}-\ref{dxh} hold. Then, for every $p>1$,
$$\left\|g_\epsilon\right\|_{L^\infty(0,T;L^p(\Tt^d))}\leq C+C\|Du^\ep\|^{2(\gamma-1)}_{L^\infty(\Tt^d\times[0,T])}.$$
\end{Proposition}The proof of Proposition \ref{regfp} is presented in Section \ref{rfk}. Then the nonlinear adjoint method (see \cite{E3}, as well as \cite{T1}) yields the following estimate:
\begin{Proposition}\label{prop}
Let $(u^\ep,m^\ep)$ be a solution of \eqref{smfg}-\eqref{itbc} and assume that A\ref{ham}-\ref{gamma} are satisfied. Then,
\begin{align*}
\left\|Du^\ep\right\|_{L^\infty(\Tt^d\times[0,T])}&\leq C+C\|g_\ep\|_{L^\infty(0,T;L^p(\Tt^d))}\\&\quad+C\|g_\ep\|_{L^\infty(0,T;L^p(\Tt^d))}\|Du^\ep\|^{2(\gamma-1)}_{L^\infty(\Tt^d\times[0,T])},
\end{align*}for some $p>1$.
\end{Proposition}Proposition \ref{prop} is proved in Section \ref{liphj}. By
combining Proposition \ref{prop} with Proposition \ref{regfp} we 
obtain the next Theorem, which gives the Lipschitz regularity $u^\ep$.
\begin{teo}\label{uLip}
Assume that A\ref{ham}-\ref{gamma} from Section \ref{assumptions} hold.
There exists a constant $C>0$ such that any
solution $(u^\epsilon,m^\epsilon)$ to \eqref{smfg}-\eqref{itbc} satisfies $\|Du^\epsilon\|_{L^\infty}<C$. 
\end{teo}
\begin{proof}
By combining Propositions \ref{regfp} and \ref{prop}, we get
\begin{align*}
\|Du^\epsilon\|_{L^\infty(\Tt^d\times[0,T])}&\leq C+C\|Du^\epsilon\|_{L^\infty(\Tt^d\times[0,T])}^{4(\gamma-1)}.
\end{align*}
Assumption A\ref{gamma} ensures that $4(\gamma-1)<1$. Hence,
a weighted Young's inequality yields the result.
\end{proof}

Once this is done, to prove Theorem \ref{main}, we need to establish additional regularity for $(u^\ep,m^\ep)$. This allows us to pass to the limit $\ep\to 0$ and to conclude that $(u,m)=\lim_{\epsilon\to 0}(u^\ep,m^\ep)$ solves the system \eqref{mfg}-\eqref{itbc} in some appropriate sense as well. Because $(u,m)$ has the same regularity of $(u^\ep,m^\ep)$, the existence of classical solutions is proven. This argument is set forth in Section \ref{liphj}. 


\section{Estimates in $L^\infty(\Tt^d\times[0,T])$}\label{ee}

Fix $x_0\in \Tt^d$ and $0\leq\tau<T$.
Based upon the ideas in \cite{E3} (see also \cite{T1}), 
we begin by introducing the linearized adjoint equation
\begin{equation}\label{adj}
\begin{cases}
\rho_t-\div(D_pH(x,Du^\epsilon)\rho)=\Delta \rho\\\rho(x,\tau)\,=\,\delta_{x_0},
\end{cases}
\end{equation}
where $\delta_{x_0}$ is the Dirac delta centered at $x_0$. 

If $g_\ep$ were bounded from below, the optimal control formulation for the Hamilton-Jacobi equation in \eqref{smfg} would immediately yield an estimate from below for the value function $u^\ep$ in $L^\infty(\Tt^d\times[0,T])$. This is not the case in the presence of a logarithmic nonlinearity. The next Proposition investigates upper bounds for the solutions of the Hamilton-Jacobi equation.

\begin{Proposition}\label{prop1}
Let $(u^\ep,m^\ep)$ be a solution of \eqref{smfg}, and assume that A\ref{ham}
is satisfied. Suppose further that $\rho$ solves \eqref{adj}. Then, for any $p, q\geq 1$ such that 
\begin{equation}\label{rest1}
\frac{1}{p}\,+\,\frac{1}{q}\,=\,1
\end{equation}and 
\begin{equation}\label{rest2}
p>\frac{d}{2},
\end{equation}
we have
\begin{equation}\label{star1}
\|u^\ep\|_{L^\infty(\Tt^d\times[0,T])}\,\leq\,C\,+\,C\|g_\ep\|_{L^\infty(0,T;L^p(\Tt^d))}\left(1\,+\,\|\rho\|_{L^1(0,T;L^q(\Tt^d))}\right).
\end{equation}
\end{Proposition}
\begin{proof}For ease of presentation, we drop the $\ep$ in this proof.
Let $0\leq \tau\leq T$.
If \eqref{rest2} is satisfied, it follows that
$$u(x_0,\tau)\,\leq\,C\,+\,C\|g\|_{L^\infty(0,T;L^p(\Tt^d))},$$see \cite{GPM2}. Furthermore,
by multiplying the first equation of \eqref{smfg} by $\rho$ and \eqref{adj} by $u$, adding them, and integrating by parts, we obtain that (using $0\leq \tau\leq T$)
\begin{align*}
u(x_0,\tau)\,&\geq\,-\,C\,-\,C\int_0^T\int_{\Tt^d}|g\rho|dxdt\\&\geq \,-\,C\,-\,\,C\|g\|_{L^\infty(0,T;L^p(\Tt^d))}\|\rho\|_{L^1(0,T;L^q(\Tt^d))},
\end{align*}where the second estimate follows from H\"older's inequality. Thus, we have proven \eqref{star1}.
\end{proof}

The next Corollary establishes a key estimate.

\begin{Corollary}\label{cor1}
Let $(u^\ep,m^\ep)$ be a solution of \eqref{smfg} and assume that A\ref{ham}-\ref{dphminush} hold. Assume further that $\rho$ solves \eqref{adj}, fix $p, q\geq 1$ such that \eqref{rest1}-\eqref{rest2} are satisfied. Then,
$$\int_0^T\int_{\Tt^d}H\rho dxdt\,\leq\,C\,+\,C\|g_\ep\|_{L^\infty(0,T;L^p(\Tt^d))}\left(1+\|\rho\|_{L^1(0,T;L^q(\Tt^d))}\right).$$
\end{Corollary}
\begin{proof}As before, we omit the $\ep$ in this proof.
Multiply the first equation of \eqref{smfg} by $\rho$ and \eqref{adj} by $u$, add them, and integrate by parts to obtain
$$c\int_0^T\int_{\Tt^d}H\rho dxdt\,\leq \,C\,+\,u^\ep(x_0,0\,)+\,C\|g_\ep\|_{L^\infty(0,T;L^p(\Tt^d))}\|\rho\|_{L^1(0,T;L^q(\Tt^d))},$$where we have used A\ref{dphminush}. Proposition \ref{prop1} yields the result.
\end{proof}

The former estimates depend both on $g_\ep$ as well as on the adjoint variable $\rho$.
In the next Section, we investigate the regularity of $\rho$ in an attempt to remove this dependence.

\section{Regularity for the adjoint variable}

The following Proposition is critical to investigating the regularity of the solutions $u^\ep$ to the regularized Hamilton-Jacobi equation in \eqref{smfg}.

\begin{Proposition}\label{prop2}
Let $(u^\ep,m^\ep)$ be a solution of \eqref{smfg}, and assume that A\ref{dphsq} holds. Let $0\leq \tau\leq T$, $\nu\in(0,1)$ and $\rho$ be a solution of \eqref{adj}.
Then,
\begin{equation}\label{star2}
\int_\tau^T\int_{\Tt^d}\left|D\left(\rho^\frac{\nu}{2}\right)\right|^2dxdt\,\leq\,C\,+\,C\left\|Du^\ep\right\|^{2\left(\gamma-1\right)}_{L^\infty(\Tt^d\times[0,T])}.
\end{equation}
\end{Proposition}
\begin{proof}As before, we omit the $\epsilon$ throughout the proof. For $0<\nu<1$, multiply \eqref{adj} by $\nu\rho^{\nu-1}$ and integrate by parts to obtain, after some elementary estimates,
and using Assumption \ref{dphsq}
$$
\int_\tau^T\int_{\Tt^d}\left|D\left(\rho^\frac{\nu}{2}\right)\right|^2dxdt\,\leq\,C+
C\int_\tau^T\int_{\Tt^d}|D_pH|^2\rho^\nu dxdt
\leq 
C+C\|Du\|^{2(\gamma-1)}_{L^\infty(\Tt^d\times[0,T])}, 
$$
because $\rho$ is a probability measure for each fixed time and $0<\nu<1$. 
%
\end{proof}

We state now a technical lemma:
\begin{Lemma}
\label{techlem1}
Suppose $q, b, \lambda\in\Rr$ satisfy
\begin{equation}
\label{techcon}
q,\,b\,\geq 1,\;\;\;0<\lambda<1,\;\;\;q< \frac{d b}{bd-2\lambda}.
\end{equation}
Denote by $2^*$ the Sobolev conjugated exponent $2^*=\frac{2d}{d-2}$.
Then there exists  $a,\,M\geq 1$, $Q\geq q$, 
$0<\kappa<1$ and $0<\tilde{\nu}<1$ such that 
\begin{equation}\label{rest3}
\frac{1}{M}=\frac{\lambda}{b},
\end{equation}
\begin{equation}\label{rest4}
\frac{1}{Q}=1-\lambda+\frac{\lambda}{a}, 
\end{equation}
\begin{equation}\label{rest5}
\frac{1}{a}=1-\kappa+\frac{2\kappa}{2^*\tilde{\nu}}, 
\end{equation}
and
\begin{equation}\label{restkey}
\frac{\kappa b}{\tilde{\nu}}\leq 1.
\end{equation}
\end{Lemma}
\begin{proof}
The Lemma is established by elementary computations that we outline next. 
Because $\lambda<1$ and $b\geq 1$, it is always possible to choose $M$ such that \eqref{rest3}
holds. We write the condition $a\geq 1$ as $\frac 1 a \leq 1$ with $\frac 1 a>0$. These two inequalities, 
together with \eqref{rest4} and \eqref{rest5} are linear in $\frac 1 a$. From this, we conclude that
a real number $a$ satisfying those conditions
exists as long as
\[
Q\geq q> \frac{d \tilde{\nu}}{\kappa \lambda (2+d (\tilde{\nu}-1))-d \tilde{\nu}}
\]
and $\tilde \nu>\frac{d-2}{d}$.
From this, the existence of $Q$ is also trivial. 
Combining this condition with \eqref{restkey}, 
the existence of $0<\kappa<1$ follows, as long as  
\begin{equation}
\label{starfish}
q < \frac{b d}{b d - 2 \lambda + d \lambda - d \lambda \tilde \nu}
\end{equation}
holds. The condition $q< \frac{d b}{bd-2\lambda}$ then results from the elimination of the remaining variable $\tilde \nu$, 
which occurs as $\tilde \nu\to 1$ in \eqref{starfish}.
 \end{proof}

The bounds in Section \ref{ee} depend on norms of $\rho$ in $L^1(0,T;L^q(\Tt^d))$, 
which are estimated in the next Lemma. 
\begin{Lemma}\label{lem1}
Let $(u^\ep,m^\ep)$ be a solution of \eqref{smfg} and $\rho$ be a solution to \eqref{adj}. Assume that A\ref{dphsq} holds and suppose $q, b, \lambda\in \Rr$ satisfy
\eqref{techcon}. Then
$$\|\rho\|_{L^1(0,T;L^q(\Tt^d))}\,\leq\,C\,+\,C\|Du^\ep\|^\frac{2\lambda(\gamma-1)}{b}_{L^\infty(\Tt^d\times[0,T])}.$$
\end{Lemma}
\begin{proof}
Since \eqref{techcon} holds, fix $M$, $Q$, $a$, $\kappa$, and $\tilde{\nu}$
 as in the statement of Lemma \ref{techlem1}.
Hence, by H\"older's inequality combined with \eqref{rest3} and \eqref{rest4}, we have
$$\|\rho\|_{L^1(0,T;L^q(\Tt^d))}\,\leq\,\|\rho\|_{L^M(0,T;L^Q(\Tt^d))}\,\leq\,\|\rho\|^{1-\lambda}_{L^\infty(0,T;L^1(\Tt^d))}\|\rho\|^{\lambda}_{L^b(0,T;L^a(\Tt^d))}. $$
Using 
H\"older's inequality once more, together with \eqref{rest5}, yields
$$\left(\int_{\Tt^d}\rho^a dx\right)^\frac{1}{a}\,\leq\,\left(\int_{\Tt^d}\rho dx\right)^{1-\kappa}\left(\int_{\Tt^d}\rho^\frac{2^*\tilde{\nu}}{2}dx\right)^\frac{2\kappa}{2^*\tilde{\nu}}.$$
We have, by Sobolev's Theorem 
$$\left(\int_{\Tt^d}\rho^\frac{2^*\tilde{\nu}}{2}\right)^\frac{2\kappa}{2^*\tilde{\nu}}\,\leq\,C\,+\,C\left(\int_{\Tt^d}\left|D\left(\rho^\frac{\tilde{\nu}}{2}\right)\right|^2dx\right)^\frac{\kappa}{\tilde{\nu}}$$
and, therefore,
$$\int_0^T\left(\int_{\Tt^d}\rho^adx\right)^\frac{b}{a}\,\leq\,C\,+\,C\int_0^T\left(\int_{\Tt^d}\left|D\left(\rho^\frac{\tilde{\nu}}{2}\right)\right|^2dx\right)^\frac{\kappa b}{\tilde{\nu}}.$$
Using \eqref{restkey} in
%
the previous computation combined with Proposition \ref{prop2}, we obtain $$\|\rho\|^\lambda_{L^b(0,T;L^a(\Tt^d))}\leq C\,+\,C\|Du^\ep\|_{L^\infty(\Tt^d\times[0,T])}^\frac{2\lambda(\gamma-1)}{b},$$which concludes the proof.
\end{proof}

\begin{Corollary}\label{cor2}
Let $(u^\ep,m^\ep)$ be a solution of \eqref{smfg}, and assume that A\ref{ham}-\ref{dphminush} hold. Assume further that $\rho$ solves \eqref{adj} and that $\lambda$, $p$, and $b$ satisfy
\begin{equation}
\label{techcon2}
0<\lambda <1, \qquad  p>\frac{d}{2\lambda}, \qquad 1<b<\frac{2\lambda p}{d}.
\end{equation}
%
%
Then,
$$\int_0^T\int_{\Tt^d}H\rho dxdt\,\leq\,C\,+\,C\|g_\ep\|_{L^\infty(0,T;L^p(\Tt^d))}\left(1+\|Du^\ep\|_{L^\infty(\Tt^d\times[0,T])}^{\frac{2\lambda(\gamma-1)}{b}}\right).$$
\end{Corollary}
\begin{proof}
The result follows by combining Corollary \ref{cor1} with Lemma \ref{lem1} and observing that \eqref{techcon}, 
\eqref{rest1}, and \eqref{rest2} are 
equivalent to \eqref{techcon2} and \eqref{rest1}. 
\end{proof}

\section{Estimates for the Fokker-Planck equation}\label{rfk}

Next, we establish several estimates concerning the integrability of solutions of the Fokker-Planck equation. 
The proof of Proposition \ref{regfp} closes this Section.

\begin{Lemma}\label{lemma1}
Let $(u^\ep,m^\ep)$ be a solution of \eqref{smfg}. Then,
$$\frac{d}{dt}\left[\ln\left(\int_{\Tt^d}\frac{1}{m+\epsilon}dx\right)\right]\leq C\left\|\left|D_pH\right|^2\right\|_{L^\infty(\Tt^d)}+C.$$
\end{Lemma}
\begin{proof}
Notice that 
\begin{align*}
\frac{d}{dt}\int_{\Tt^d}\frac{1}{m+\epsilon}dx=-\int_{\Tt^d}\frac{\div\left(D_pHm\right)}{(m+\epsilon)^2}dx-\int_{\Tt^d}\frac{\Delta m}{(m+\epsilon)^2}dx.
\end{align*}Integration by parts yields
$$-\int_{\Tt^d}\frac{\div\left(D_pHm\right)}{(m+\epsilon)^2}dx=-2\int_{\Tt^d}\frac{D_pHmDm}{(m+\epsilon)^\frac{3}{2}(m+\epsilon)^\frac{3}{2}}$$and 
$$-\int_{\Tt^d}\frac{\Delta m}{(m+\epsilon)^2}=-2\int_{\Tt^d}\frac{\left|Dm\right|^2}{(m+\epsilon)^3}dx.$$
Hence,  for some $C, c>0$, 
\begin{align*}
\frac{d}{dt}\int_{\Tt^d}\frac{1}{m+\epsilon}dt&\leq C\int_{\Tt^d}\frac{|D_pH|^2m^2}{(m+\epsilon)^3}dx-c\int_{\Tt^d}\frac{|Dm|^2}{(m+\epsilon)^3}dx\\&\leq C\left\|\left|D_pH\right|^2\right\|_{L^\infty(\Tt^d)}\int_{\Tt^d}\frac{1}{m+\epsilon}dx.
\end{align*}Consequently,
$$\frac{d}{dt}\left[\ln\left(\int_{\Tt^d}\frac{1}{m+\epsilon}dx\right)\right]\leq C\left\|\left|D_pH\right|^2\right\|_{L^\infty(\Tt^d)}.$$
\end{proof}
\begin{Lemma}\label{lemma2}
Let $m:\Tt^d\to \Rr^+_0$ be integrable with $\int_{\Tt^d} m=1$. Then,
$$\int_{\Tt^d}|\ln (m+\epsilon)|^pdx\leq C+C\int_{m+\ep\leq 1}\left(\ln\frac{1}{m+\epsilon}\right)^pdx.$$
\end{Lemma}
\begin{proof}
We have $$\int_{\Tt^d}|\ln (m+\epsilon)|^pdx=\int_{m+\ep\leq 1}\left(\ln\frac{1}{m+\epsilon}\right)^pdx+\int_{m+\ep>1}\left[\ln (m+\epsilon)\right]^pdx.$$
Because
$\ln (m+\epsilon)\leq C_\delta(m+\epsilon)^\delta$ for every $\delta>0$, provided $(m+\epsilon)>1$, we conclude
$$\int_{\Tt^d}|\ln (m+\epsilon)|^pdx=\int_{m+\ep\leq 1}\left(\ln\frac{1}{m+\ep}\right)^pdx+C.$$
\end{proof}
\begin{Lemma}\label{lemma3}
There exists $0<A$, depending solely on $p$, such that $\left(\ln z\right)^p$is a concave function for $z>\frac{1}{A}$.
\end{Lemma}
\begin{proof}
A straightforward computation implies
$$\left[\left(\ln z\right)^p\right]''=\frac{p\left(\ln z\right)^{p-2}}{z^2}\left[p-1-\ln z\right].$$ For $z>e^{p-1}$, we have
that$$\left[\left(\ln z\right)^p\right]''<0.$$ Therefore, the result holds for $A=e^{1-p}$.
\end{proof}
\begin{Lemma}\label{lemma4}
Let $(u^\ep,m^\ep)$ be a solution of \eqref{smfg}. Then, 
$$\int_{m+\ep\leq 1}\left(\ln\frac{1}{m(x,\tau)+\ep}\right)^pdx\leq C+C\left\|\left|D_pH\right|^2\right\|_{L^\infty(\Tt^d\times[0,T])}^p.$$
\end{Lemma}
\begin{proof}
Clearly, 
\begin{align*}
\int_{m+\ep\leq 1}\left(\ln\frac{1}{m+\ep}\right)^pdx&=\int_{A\leq m+\ep\leq 1}\left(\ln\frac{1}{m+\ep}\right)^pdx\\&\quad+\int_{m+\ep<A}\left(\ln\frac{1}{m+\ep}\right)^pdx,
\end{align*}for every $0<A<1$.
Select $A$ as in Lemma \ref{lemma3}. 
First note that
$$\int_{A\leq m+\ep\leq 1}\left(\ln\frac{1}{m+\ep}\right)^pdx\leq C\max_{A\leq m+\ep\leq 1}\left|\ln\frac{1}{m+\ep}\right|^p\leq C.$$
Define $\Psi(z)=(\ln z)^p$, for $z>\frac 1 A$, and extend it continuously and linearly for $z<\frac 1 A$. It is
possible to do so in such a way that $\Psi$ is globally concave and increasing. 

Then Jensen's inequality leads to 
\begin{align*}
\frac{1}{\left|\left\lbrace m+\ep\leq A\right\rbrace\right|}\int_{m+\ep\leq A}\left(\ln\frac{1}{m+\ep}\right)^pdx&=
\frac{1}{\left|\left\lbrace m+\ep\leq A\right\rbrace\right|}\int_{m+\ep\leq A}\Psi\left(\frac{1}{m+\ep}\right)dx\\
&\leq \Psi\left(\frac{1}{\left|\left\lbrace m+\ep\leq A\right\rbrace\right|}\int_{m+\ep<A}\frac{1}{m+\ep}\right). 
\end{align*}
Since $\Psi$ is increasing we conclude that 
\[
\frac{1}{\left|\left\lbrace m+\ep\leq A\right\rbrace\right|}\int_{m+\ep\leq A}\left(\ln\frac{1}{m+\ep}\right)^pdx\leq \Psi\left(\frac{1}{\left|\left\lbrace m+\ep\leq A\right\rbrace\right|}\int_{\Tt^d}\frac{1}{m+\ep}\right).
\]
Now, there are two cases, either 
\[
\frac{1}{\left|\left\lbrace m+\ep\leq A\right\rbrace\right|}\int_{\Tt^d}\frac{1}{m+\ep}<\frac 1 A, 
\]
or 
\[
\frac{1}{\left|\left\lbrace m+\ep\leq A\right\rbrace\right|}\int_{\Tt^d}\frac{1}{m+\ep}>\frac 1 A.  
\]
In the former case, we have
\[
\int_{m+\ep\leq A}\left(\ln\frac{1}{m+\ep}\right)^pdx
\leq \Psi\left(\frac 1 A\right).
\]
From the latter, it follows that
\begin{align*}
\Psi\left(\frac{1}{\left|\left\lbrace m+\ep\leq A\right\rbrace\right|}\int_{\Tt^d}\frac{1}{m+\ep}\right)
=\left[\ln\left(\frac{1}{\left|\left\lbrace m+\ep\leq A\right\rbrace\right|}\right)+
\ln\left(\int_{\Tt^d}\frac{1}{m+\ep}\right)\right]^p.
\end{align*}
Therefore, since $|\{m+\epsilon\leq A\}|\leq 1$,
\begin{align*}
\int_{m+\ep\leq A}\left(\ln\frac{1}{m+\ep}\right)^pdx&\leq 
\Psi\left(\frac 1 A\right)+
C_p\left[\ln\left(\int_{\Tt^d}\frac{1}{m+\ep}\right)\right]^p\\&\quad+C_p\left|\left\lbrace m+\ep\leq A\right\rbrace\right|\left[\ln\left(\frac{1}{\left|\left\lbrace m+\ep\leq A\right\rbrace\right|}\right)\right]^p.
\end{align*}Because $$\frac{1}{\left|\left\lbrace m+\ep\leq A\right\rbrace\right|}\,\geq\,1,$$
it follows that $$\ln\left(\frac{1}{\left|\left\lbrace m+\ep\leq A\right\rbrace\right|}\right)\,\leq\,C_\delta
\left(\frac{1}{\left|\left\lbrace m+\ep\leq A\right\rbrace\right|}\right)^\delta,$$ for every $\delta>0$. By choosing $\delta=\frac{1}{p}$ one obtains
\begin{align*}
\int_{m+\ep\leq A}\left(\ln\frac{1}{m+\ep}\right)^pdx&\leq C+C\left[\ln\left(\int_{\Tt^d}\frac{1}{m+\ep}\right)\right]^p\\&\quad+C\frac{\left|\left\lbrace m+\ep\leq A\right\rbrace\right|}{\left|\left\lbrace m+\ep\leq A\right\rbrace\right|},
\end{align*}that is,
\begin{align*}
\int_{m+\ep\leq A}\left(\ln\frac{1}{m+\ep}\right)^pdx&\leq C+C\left[\ln\left(\int_{\Tt^d}\frac{1}{m+\ep}\right)\right]^p,
\end{align*}which concludes the proof, using Lemma \ref{lemma1}. 
\end{proof}

We end this Section with the proof of Proposition \ref{regfp}.

\begin{proof}[Proof of Proposition \ref{regfp}.]
By combining Lemmas \ref{lemma2} and \ref{lemma4} one obtains that $$\int_{\Tt^d}|\ln (m_\epsilon+\epsilon)|^pdx\leq C+C\left\||D_pH|^2\right\|^p_{L^\infty(\Tt^d\times[0,T])}.$$
The Proposition is then implied by A\ref{dphsq}.
\end{proof}

\section{Lipschitz regularity for the Hamilton-Jacobi equation}\label{liphj}

In this Section, we obtain estimates for $Du^\ep$ in $L^\infty(\Tt^d\times[0,T])$, uniformly in $\epsilon$. We begin with a technical lemma, followed by the proof of Proposition \ref{prop}. This Section ends with the proof of Theorem \ref{uLip}.

\begin{Lemma}\label{lem61}
For $d>2$
there exist real numbers $\lambda$, $b$, $p$, $\tilde q$, $\theta$, and $\bar \nu$
such that \eqref{techcon2}, 
\begin{equation}\label{rest6}
\frac{1}{p}+\frac{1}{\tilde{q}}=\frac{1}{2},\qquad \tilde q\geq 1 
\end{equation}
\begin{equation}\label{rest7}
\frac{1}{\tilde{q}\left(\frac{2-\bar{\nu}}{2}\right)}=1-\theta+\frac{2\theta}{2^*\bar{\nu}}, \qquad 0<\theta<1
\end{equation}
\begin{equation}\label{rest8}
\theta=\frac{\bar{\nu}}{2-\bar{\nu}},\qquad \mbox{and}\qquad 0<\bar \nu<1
\end{equation}
hold simultaneously. 
\end{Lemma}
\begin{proof}
The Lemma is established by elementary computations that are outlined next. The identities in \eqref{rest6}-\eqref{rest8}
can be solved directly for $\tilde q$, $\theta$ and $\bar \nu$. Thus, the inequalities in  \eqref{rest6}-\eqref{rest8}
are translated into
\[
\frac{2p}{p-2}\geq 1, \qquad 0<\frac{d-p+d p}{p-d+dp}<1, \qquad 0<\frac{d-p+d p}{dp}<1.
\]
Since $d>2$, all the previous inequalities hold for $p>d$.
The parameter $\lambda$ can be chosen arbitrarily satisfying $0<\lambda<1$. 
Then,
if $p$ is chosen large enough so that 
$p>\max\{\frac{d}{2\lambda}, d\}$, the existence of
$b$ satisfying \eqref{techcon2} is immediate. 
\end{proof}

\begin{proof}[Proof of Proposition \ref{prop}.]For ease of notation, we omit $\epsilon$ throughout the proof. Choose $\lambda$, $b$, $p$, $\tilde q$, $\theta$ and $\bar \nu$ as in Lemma \ref{lem61}. Let $\rho$ be a solution of \eqref{adj}.
Fix a unit vector $\xi$ and differentiate the first equation in \eqref{smfg} in the $\xi$ direction. Multiply it by $\rho$ and \eqref{adj} by $u_\xi$; by adding them and integrating by parts, it follows that 
$$u_\xi(x_0,\tau)=\int_\tau^T\int_{\Tt^d}-D_\xi H\rho+g_\xi\rho dxdt+\int_{\Tt^d}(u_T)_\xi\rho(x,T)dx.$$
Assumption A\ref{dxh} implies 
\begin{align*}
\left|\int_\tau^T\int_{\Tt^d}-D_\xi H\rho dxdt\right|&\leq C\,+\,C\int_0^T\int_{\Tt^d}H\rho dxdt\\&\leq C\,+\,C\|g_\ep\|_{L^\infty(0,T;L^p(\Tt^d))}\left(1+\|Du^\ep\|_{L^\infty(\Tt^d\times[0,T])}^{\frac{2\lambda(\gamma-1)}{b}}\right),
\end{align*}
by Corollary \ref{cor2}. Moreover, $$\left|\int_{\Tt^d}(u_T)_\xi\rho(x,T)dx\right|\leq C$$because it depends only on the terminal data. It remains to bound the term $$\int_0^T\int_{\Tt^d}g_\xi\rho dx.$$
Integrating by parts gives the following estimate:
$$\left|\int_0^T\int_{\Tt^d}g_\xi\rho dx\right|\leq C\|g\|_{L^\infty(0,T;L^p(\Tt^d))}\|\rho^{1-\frac{\bar{\nu}}{2}}\|_{L^2(0,T;L^{\tilde{q}}(\Tt^d))}\|D\rho^\frac{\bar{\nu}}{2}\|_{L^2(\Tt^d\times[0,T])},$$
where $\tilde q$ is given by \eqref{rest6}. 
From estimate \eqref{star2} in Proposition \ref{prop2}, it follows that
$$\|D\rho^\frac{\bar{\nu}}{2}\|_{L^2(\Tt^d\times[0,T])}\,\leq\, C\,+\,C\|Du\|_{L^\infty(\Tt^d\times[0,T])}^{\gamma-1}.$$ Moreover, 
$$\left(\int_{\Tt^d}\rho^{\tilde{q}\left(\frac{2-\bar{\nu}}{2}\right)}\right)^\frac{1}{\tilde{q}\left(\frac{2-\bar{\nu}}{2}\right)}\leq \left(\int_{\Tt^d}\rho\right)^{1-\theta}\left(\int_{\Tt^d}\rho^\frac{2^*\bar{\nu}}{2}dx\right)^\frac{2\theta}{2^*\bar{\nu}},$$provided \eqref{rest7} holds.
Then Sobolev's Theorem yields 
$$\left(\int_{\Tt^d}\rho^{\tilde{q}\left(\frac{2-\bar{\nu}}{2}\right)}\right)^\frac{1}{\tilde{q}\left(\frac{2-\bar{\nu}}{2}\right)}\leq C+C\left(\int_{\Tt^d}\left|D\left(\rho^\frac{\bar{\nu}}{2}\right)\right|^2\right)^\frac{\theta}{\bar{\nu}}.$$
Consequently, 
$$\left(\int_{\Tt^d}\rho^{\tilde{q}\left(\frac{2-\bar{\nu}}{2}\right)}\right)^\frac{2}{\tilde{q}}\leq C+C\left(\int_{\Tt^d}\left|D\left(\rho^\frac{\bar{\nu}}{2}\right)\right|^2\right)^\frac{(2-\bar{\nu})\theta}{\bar{\nu}}.$$By setting
$\theta$ as in \eqref{rest8}, 
and recurring to Proposition \ref{prop2} one obtains that $$\|\rho^{1-\frac{\bar{\nu}}{2}}\|_{L^2(0,T;L^{\tilde{q}}(\Tt^d))}\leq C+C\|Du\|_{L^\infty(\Tt^d\times[0,T])}^{\gamma-1}.$$

By gathering the previous computation, we conclude that
\begin{align*}
\left|u_\xi(x,\tau)\right|&\leq C+C\|g\|_{L^\infty(0,T;L^p(\Tt^d))}\left(1+\|Du\|_{L^\infty(\Tt^d\times[0,T])}^\frac{2\lambda(\gamma-1)}{b}\right)\\&\quad+C\|g\|_{L^\infty(0,T;L^p(\Tt^d))}\left(1+\|Du\|^{2(\gamma-1)}_{L^\infty(\Tt^d\times[0,T])}\right),
\end{align*}which becomes 
\begin{align*}
\left|u_\xi(x,\tau)\right|&\leq C+C\|g\|_{L^\infty(0,T;L^p(\Tt^d))}+C\|g\|_{L^\infty(0,T;L^p(\Tt^d))}\|Du\|^{2(\gamma-1)}_{L^\infty(\Tt^d\times[0,T])},
\end{align*}once we take into account that
$$\frac{2\lambda(\gamma-1)}{b}<2(\gamma-1).$$
\end{proof}

%

\begin{proof}[Proof of Theorem \eqref{main}]
We first notice that $\ln(m^\epsilon +\epsilon)$ is in $L^p$, uniformly in $\epsilon$. On the other hand, $Du^\epsilon$ is bounded in $L^\infty$, uniformly in $\epsilon$. Therefore, from standard regularity theory for the heat equation, we have $D^2u^\epsilon$ uniformly bounded in $L^p$, for every $1<p<\infty$. 

On the other hand, consider the Hopf-Cole transformation $v^\epsilon\doteq \ln(m^\epsilon+\epsilon)$. It follows from elementary computations that $$v_t^\epsilon-D_pHDv^\epsilon-\div(D_pH)=|Dv^\epsilon|^2+\Delta v,$$
which implies that $v^\epsilon$ is bounded in $L^\infty$, uniformly in $\epsilon$ (see \cite{GPM2}).
Consequently, $m^\ep$ is also bounded by below, uniformly in $\epsilon$. Once this is established, because $\ln (m+\epsilon)$ is uniformly bounded 
by below and has sub-polynomial growth, the techniques in \cite{GPM2} can be applied without any substantial change. 
Therefore, we conclude that $u^\ep$ is smooth, with norms uniformly bounded in every Sobolev space. From this, it follows that, through some subsequence, $u^\ep\to u$ in any Sobolev space. As a consequence, we obtain that $m^\ep\to m$ in the strong sense in any Sobolev space. Hence, the limit $(u,m)$ is a classical solution of \eqref{mfg}-\eqref{itbc}, and the proof is complete.
\end{proof}



\bibliography{mfg}
\bibliographystyle{plain}

\end{document}